\documentclass[letterpaper,11pt]{amsart}

\usepackage[all]{xy}                        %

\CompileMatrices                            

\UseTips                                    

\input xypic
\usepackage[bookmarks=true]{hyperref}       

\usepackage{amssymb,latexsym,amsmath,amscd}
\usepackage{xspace}

\usepackage{graphicx}

\reversemarginpar

\theoremstyle{plain}
\newtheorem{theorem}{Theorem}[section]
\newtheorem*{theorem*}{Theorem}
\newtheorem{proposition}[theorem]{Proposition}
\newtheorem{corollary}[theorem]{Corollary}
\newtheorem{lemma}[theorem]{Lemma}

\theoremstyle{definition}
\newtheorem{definition}[theorem]{Definition}

\newtheorem{remark}[theorem]{Remark}

\newcommand{\enm}[1]{\ensuremath{#1}}          %
\newcommand{\op}[1]{\operatorname{#1}}
\newcommand{\cal}[1]{\mathcal{#1}}

\newcommand{\CC}{\enm{\mathbb{C}}}

\newcommand{\PP}{\enm{\mathbb{P}}}

\newcommand{\Ee}{\enm{\cal{E}}}
\newcommand{\Ff}{\enm{\cal{F}}}

\newcommand{\Ii}{\enm{\cal{I}}}

\newcommand{\Ll}{\enm{\cal{L}}}
\newcommand{\Mm}{\enm{\mathfrak{M}}}

\newcommand{\Oo}{\enm{\cal{O}}}

\renewcommand{\phi}{\varphi}
\renewcommand{\theta}{\vartheta}
\renewcommand{\epsilon}{\varepsilon}

\newcommand{\Ext}{\op{Ext}}

\newcommand{\old}[1]{}

\begin{document}

\title[Cubic Symmetroids and vector bundles on a quadric surface]{Cubic Symmetroids\\and vector bundles on a quadric surface}

\author{Sukmoon Huh}
\address{Department of Mathematics, Sungkyunkwan University \\
Cheoncheon-dong, Jangan-gu \\
Suwon 440-746, Korea}
\email{sukmoonh@skku.edu}
\thanks{The author is supported by Basic Science Research Program 2012-0002904 through NRF funded by MEST}

\maketitle

\begin{abstract}
We investigate the jumping conics of stable vector bundles $\Ee$ of rank 2 on a smooth quadric surface $Q$ with the Chern classes $c_1=\Oo_Q(-1,-1)$ and $c_2=4$ with respect to the ample line bundle $\Oo_Q(1,1)$. We describe the set of jumping conics of $\Ee$, a cubic symmetroid in $\PP_3$, in terms of the cohomological properties of $\Ee$. As a consequence, we prove that the set of jumping conics, $S(\Ee)$, uniquely determines $\Ee$. Moreover we prove that the moduli space of such vector bundles is rational.  
\end{abstract}

\section{Introduction}
Throughout the article, our base field is $\CC$, the complex numbers.

Let $Q$ be a smooth quadric in $\PP_3=\PP (V)$, where $V$ is a 4-dimensional vector space over complex numbers $\CC$, and $\Mm(k)$ be the moduli space of stable vector bundles of rank 2 on $Q$ with the Chern classes $c_1=\Oo_Q(-1,-1)$ and $c_2=k$ with respect to the ample line bundle $\Ll=\Oo_Q(1,1)$. $\Mm (k)$ forms an open Zariski subset of the projective variety $\overline{\Mm}(k)$ whose points correspond to the semi-stable sheaves on $Q$ with the same numerical invariants. The Zariski tangent space of $\Mm(k)$ at $\Ee$, is naturally isomorphic to $H^1(Q, \mathcal{E}nd (\Ee))$ \cite{LePotier} and so the dimension of $\Mm (k)$ is equal to $h^1(Q, \mathcal{E}nd (\Ee))=4k-5$, since $\Ee$ is simple.

In \cite{Huh}, we define the jumping conics of $\Ee\in \Mm(k)$ as a point in $\PP_3^*$ and prove that the set of jumping conic is a symmetric determinantal hypersurface of degree $k-1$ in $\PP_3^*$. It enables us to consider a morphism
$$S : \Mm(k) \rightarrow |\Oo_{\PP_3^*}(k-1)|\simeq \PP_N.$$
The conjecture is that the general $\Ee\in \Mm(k)$ is uniquely determined by $S(\Ee)$. It is true in the case of $k\leq 3$ \cite{Huh}. 

In this article, we prove that the conjecture is true in the case of $k=4$. For a stable vector bundle $\Ee$ in $\Mm(4)$, $S(\Ee)$ is a cubic symmetroid surface, i.e. a symmetric determinantal cubic hypersurface in $\PP_3^*$. In terms of short exact sequences that $\Ee$ admits, we can obtain the relation between the singularity of $S(\Ee)$ and the dimension of cohomology of the restriction of $\Ee$ to its hyperplane section. It turns out that $S(\Ee)$ has exactly 4 singular points. It eventually enables us to prove the rationality of $\Mm(4)$. 

In the last part of the article, we give a brief description of $S(\Ee)$ for non-general vector bundles of $\Mm (4)$. 

We will denote the dimension of the cohomology $H^i(X, \Ff)$ for a coherent sheaf $\Ff$ on $X$ by $h^i(X,\Ff)$, or simply by $h^i(\Ff)$ if there is no confusion.

\section{Preliminaries}
Let $Q$ be a smooth quadric surface isomorphic to $\PP(V_1) \times \PP (V_2)$ for two 2-dimensional vector spaces $V_1$ and $V_2$. Then it is embedded into $\PP_3\simeq \PP (V)$ by the Segre map, where $V=V_1\otimes V_2$. Let us denote $f^*\Oo_{\PP_1}(a) \otimes g^*\Oo_{\PP_1}(b)$ by $\Oo_Q(a,b)$ and $\Ee \otimes \Oo_Q(a,b)$ by $\Ee(a,b)$ for coherent sheaves $\Ee$ on $Q$, where $f$ and $g$ are the projections from $Q$ to each factors. Then the canonical line bundle $K_Q$ of $Q$ is $\Oo_Q(-2,-2)$.

\begin{lemma}
We have
$$H^i (Q, \Oo_Q(a, a+b))=\left\{
                                           \begin{array}{ll}
                                             0, & \hbox{if $a=-1$;} \\
                                             H^i (\PP_1, \Oo_{\PP_1}(a+b)^{\oplus (a+1)}), & \hbox{if $a\geq 0$}
                                           \end{array}
                                         \right.$$
The other cases can be induced from the above by the Serre duality and symmetry. 
\end{lemma}

\begin{proof}
Note that $Q$ is the projectivization of the vector bundle $\Oo_{\PP_1}(1)^{\oplus 2}$ over $\PP_1$. Then the proof follows from the degeneration of the Leray spectral sequence
$$H^i(\PP_1, R^jf_*\Oo_Q (a,a+b)) \Rightarrow H^{i+j} (Q, \Oo_Q(a,a+b)).$$
\end{proof}

Now let us denote the ample line bundle $\Oo_Q(1,1)$ by $\Ll$. 

\begin{definition}
A torsion free sheaf $\Ee$ of rank $r$ on $Q$ is called \textit{stable} (resp. \textit{semi-stable}) with respect to $\Ll$ if
$$\frac{\chi(\Ff \otimes \Ll^{\otimes m})}{r'} < (resp. \leq ) \frac{\chi(\Ee\otimes \Ll^{\otimes m})}r,$$
for all non-zero subsheaves $\Ff\subset \Ee$ of rank $r'$.
\end{definition}

Let $\overline{\Mm} (k)$ be the moduli space of semi-stable sheaves of rank 2 on $Q$ with the Chern classes $c_1=\Oo_Q(-1,-1)$ and $c_2=k$ with respect to $\Ll=\Oo_Q(1,1)$. The existence and the projectivity of $\overline{\Mm}(k)$ is known in \cite{Gieseker} and it has an open Zariski subset $\Mm(k)$ which consists of the stable vector bundles with the given numeric invariants. By the Bogomolov theorem \cite{LePotier}, $\Mm (k)$ is empty if $4k < c_1^2=2$ and in particular, we can consider only the case of $k\geq 1$. The dimension of $\Mm (k)$ can be computed to be $h^1(Q, \mathcal{E}nd (\Ee))=4k-5$.

Note that $\Ee \simeq \Ee^*(-1,-1)$ and by the Riemann-Roch theorem \cite{Hartshorne}, we have
$$\chi_{\Ee}(m):=\chi(\Ee(m,m))=2m^2+2m+1-k,$$
for $\Ee\in \overline{\Mm}(k)$.

For a hyperplane section $H$ of $\PP_3$, we define $C_H:=Q\cap H$ to be a conic on $H$. 
\begin{definition}
The conic $C_H$ is called to be a {\it jumping conic} if $h^0(\Ee|_{C_H})\geq 1$ and let 
$$S(\Ee):=\{ H \in \PP_3^* ~|~ h^0(\Ee|_{C_H})\geq 1\}.$$
\end{definition}

When $C_H$ is a smooth conic on $H$,  it is a jumping conic if the vector bundle $\Ee$ splits non-generically over it. 

\begin{theorem}\cite{Huh}
For $\Ee\in \Mm(k)$, $S(\Ee)$ is a symmetric determinantal hypersurface of degree $k-1$ in $\PP_3^*$ and it has a singular point at $H\in \PP_3^*$ if and only if $h^0(\Ee|_{C_H})\geq 2$. 
\end{theorem}

It enables us to consider a morphism 
$$S: \Mm(k) \rightarrow |\Oo_{\PP_3^*} (k-1)|\simeq \PP_N,$$
where $N= {k+2 \choose 3}-1$. 

In \cite{Huh}\cite{Huh2}, the cases of $k=2,3$ are dealt in detail. For example, when $k=2$, the morphism $S$ extends to an isomorphism from $\overline{\Mm}(2) \rightarrow \PP_3$ and $\Mm(2)$ is isomorphic to $\PP_3\backslash Q$.  In particular, $S(\Ee)$ determines uniquely $\Ee\in \Mm(2)$. An analogue of this result turns out to be true in the case of $k=3$.

\section{Results}
From now on, we will investigate $S(\Ee)$ for $\Ee\in \Mm(4)$, which is now a {\it cubic symmetroid surface}, i.e. a symmetric determinantal cubic surface in $\PP_3^*$. Note that a nonsingular cubic surface cannot be symmetrically determinantal \cite{D}.  

Since $\chi_{\Ee}(1)=1$ and $\Ee$ is stable, it admits an exact sequence
\begin{equation}\label{seq1}
0\rightarrow \Oo_Q \rightarrow \Ee(1,1) \rightarrow \Ii_Z(1,1) \rightarrow 0,
\end{equation}
where $Z$ is a zero-dimensional subscheme of $Q$ with length 4 and $\Ii_Z(1,1)$ is the tensor product of the ideal sheaf of $Z$ and $\Oo_Q(1,1)$. Let us assume that $Z$ is in general position, then we can check that $h^0(\Ee(1,1))=1$. In particular, $Z$ is uniquely determined by $\Ee$. Note that 
$$\PP \Ext^1(\Ii_Z(1,1), \Oo_Q)\simeq \PP H^0(\Oo_Z)^* \simeq \PP_3.$$
A general point in this family of extensions corresponds to a stable vector bundle \cite{C} and so $\Mm(4)$ is birational to a $\PP_3$-bundle over the Hilbert scheme $Q^{[4]}$ of zero-dimensional subscheme of $Q$ with length 4. It is consistent with the fact that the dimension of $\Mm(4)$ is $11$. Note that $Q^{[4]}$ is a resolution of singularity of $S^4Q$, the 4th symmetric power of $Q$, and in particular it is 8-dimensional \cite{Nakajima}. 

Assume that $Z$ is not contained in any hyperplane section. If $|Z\cap H|=3$, for a hyperplane section $H$ of $\PP_3^*$, we obtain the following exact sequence by tensoring the sequence (\ref{seq1}) with $\Oo_{C_H}$:
$$0\rightarrow  \Oo_{C_H}(p) \rightarrow \Ee|_{C_H} \rightarrow \Oo_{C_H}(-3p) \rightarrow 0,$$
where $p$ is a point on $C_H$. Since $h^0(\Ee|_{C_H})=2$, $H$ is a singular point of $S(\Ee)$. 

Similarly we can prove that $H$ is a non-singular point of $S(\Ee)$ if $|Z \cap H|=2$, and not a point of $S(\Ee)$ if $|Z \cap H|=1$. If $H$ does not contain any point of $Z$, then we have the sequence 
$$0\rightarrow \Oo_{C_H}(-2p) \rightarrow \Ee|_{C_H} \rightarrow \Oo_{C_H} \rightarrow 0$$
and the dimension of $\Ext^1(\Oo_{C_H}, \Oo_{C_H}(-2p))$ is 1. So it is not obvious whether $H$ is a point of $S(\Ee)$ or not. But if it is a point of $S(\Ee)$, it would not be a singular point. So there are exact 4 singular points of $S(\Ee)$. 

Let us fix a line $l$ in $\PP_3$.   Let us consider the $\PP_1$-family of hyperplanes of $\PP_3$ that contains $l$, say $\{ H_s\}$, $s\in \PP_1$. If $l$ passes through two points of $Z$, then there are two hyperplanes in the family $\{H_s\}$ that contain three points of $Z$. In other words, the projective line in $\PP_3^*$ corresponding to the family intersects with $S(\Ee)$ at two singular points. Since $S(\Ee)$ is a cubic hypersurface, so the projective line must be contained in $S(\Ee)$. Summarizing the argument so far, we obtain:

\begin{proposition}\label{prop}
For a general vector bundle $\Ee$ in $\Mm(4)$, there are exactly 4 singular points of $S(\Ee)$ and exactly 6 lines contained in $S(\Ee)$.
\end{proposition}
\begin{proof}
The 6 lines connecting two of the 4 singular points are contained in $S(\Ee)$ and so it is enough to check that they are all. 
Let $Z=\{p_1, \cdots, p_4\}$ and denote the line connecting $p_i$ and $p_j$ by $l_{ij}$. 
Let us assume that there exists a line $l$ which is different from $l_{ij}$. As the first  case, let us assume that $l$ does not intersect with $l_{ij}$. If $\pi : \PP_3^* \dashrightarrow \PP_2^*$ is the projection from $p_1$, then the images of $l$ and $l_{ij}, i,j\not=1$ intersect. It implies that $l$ and $l_{ij}$ intersect for $i,j\not=2$. But this is impossible since the plane containing $p_2,p_3,p_4$ would contain $l$ and so . 

The case when $l$ intersects with $l_i$ can be shown to be impossible in a similar way.
\end{proof}

\begin{remark}
When we consider a $\PP_2$-family of hyperplanes of $\PP_3$ that contains a point of $Z$, the intersection of $\PP_2$ with $S(\Ee)$ is a cubic plane curve. And in fact, there are 3 hyperplanes in this family, that contain 3 points of $Z$. Thus the intersection of the $\PP_2$-family with $S(\Ee)$ is the union of three lines.  
\end{remark}
Conversely, let us consider a cubic hypersurface $S_3$ in $\PP_3^*$ with 4 singular points, say $H_1, \cdots, H_4\subset \PP_3$. Then $H_i$'s are 4 hyperplanes of $\PP_3$ in general position. If $S_3$ is equal to $S(\Ee)$ for some $\Ee \in \Mm(4)$ with the exact sequence (\ref{seq1}), then there are 3 points of $Z$ on each $H_i$. The intersection of $C_{H_1}$ with $H_i$, $i=2,3,4$ is two points of $Z$ and so 3 points of $Z$ are determined. The last point is just the intersection of $H_2, H_3$ and $H_4$.

\begin{theorem}
The morphism $S: \Mm(4) \rightarrow |\Oo_{\PP_3^*}(3)|$ is generically one-to-one. In other words, the set of jumping conics of $\Ee\in \Mm(4)$, uniquely determines $\Ee$ in general. 
\end{theorem}
\begin{proof}
It is enough to check that for two different stable vector bundles $\Ee$ and $\Ee'$ that fit into the sequence (\ref{seq1}) with the same $Z$, $S(\Ee)$ and $S(\Ee')$ are different. From the previous argument, they have the same singular points. Now, $\Ee$ and $\Ee'$ are in the extension family $\Ext^1 (\Ii_Z(1,1), \Oo_Q)$, which is isomorphic to $H^1(\Ii_Z(-1,-1))^*$. From the short exact sequence
$$0\rightarrow \Ii_Z(-1,-1) \rightarrow \Ii_Z \rightarrow \Oo_{C_H} \rightarrow 0,$$
where $C_H$ is a smooth conic that does not intersect with $Z$, we have
$$0\rightarrow H^1(\Ii_Z)^* \rightarrow H^1(\Ii_Z(-1,-1))^* \stackrel{res}{\rightarrow} H^0(\Oo_{C_H})^* \rightarrow 0.$$
Here, the map $`res'$ sends to $\Ee$ to $\Ee|_{C_H}$. Note that $H^1(\Ii_Z)^*$ is a corank 1-subspace of $H^1(\Ii_Z(-1,-1))^*$. 

If we choose $H$ properly so that the image of $H^1(\Ii_Z)^*$ contains $\Ee$, but not $\Ee'$, then their splitting will be different. To be precise, we have
$$\Ee|_{C_H}=\Oo_{C_H}(-2p)\oplus \Oo_{C_H}~~~,~~~\Ee'|_{C_H}=\Oo_{C_H}(-p)^{\oplus 2},$$
where $p$ is a point on $C_H$. In particular, $S(\Ee)$ and $S(\Ee')$ are different. 
\end{proof}

In fact, the argument in the proof of (\ref{prop}) can be applied to any symmetric determinantal cubic hypersurface with 4 singular points, we obtain the following:

\begin{corollary}
$\Mm(4)$ is biratonal to the variety of the symmetric determinantal cubic hypersurfaces with 4 singular points whose corresponding hyperplanes have 4 intersection points on $Q$.
\end{corollary}
\begin{proof}
It is known in \cite{D} that cubic surfaces with 4 rational double points are projectively isomorphic to the {\it Cayley 4-nodal cubic surface}, which is a cubic surface with 4 nodal points defined by 
$$t_0t_1t_2+t_0t_1t_3+t_0t_2t_3+t_1t_2t_3=\det \left(
               \begin{array}{ccc}
                 t_0 & 0 & t_2\\
                 0 & t_1 & -t_2\\
                 -t_3 & t_3 & t_2+t_3
               \end{array} 
             \right),
            $$
which has 4 nodal points $[1,0,0,0], [0,1,0,0], [0,0,1,0], [0,0,0,1]$. 
It means that we have a 3-dimensional family of cubic symmetroids for each fixed 4 points as singularities. Here $3=\dim \mathrm{PGL(4)}- \dim (\PP_3^{[4]})$.  So the assertion follows automatically from the previous theorem because the dimension of the variety of the cubic symmetroids in the assertion is $11=\dim (\mathrm{PGL}(4))-4$, which is the dimension of $\Mm(4)$.
\end{proof}

\begin{theorem}
$\Mm(4)$ is rational.
\end{theorem}
\begin{proof}
Let us prove that the variety $Y$ of the cubic symmetroids with 4 singular points whose corresponding hyperplanes have 4 intersection points on $Q$, is rational. First of all, the variety $X$ of cubic symmetroids with 4 singular points, generically has a $\PP_3$-bundle structure over $\PP_3^{[4]}$ and it is transitively acted by $\mathrm{PGL}(4)$. Thus $X$ is rational and we have a dominant map $\pi  : X \dashrightarrow \PP_3^{[4]}$ to a rational variety $\PP_3^{[4]}$. Since $Y$ is a subvariety of $X$ that is generically a $\PP_3$-bundle over $Q^{[4]}$ from $\pi$ and $Q^{[4]}$ is rational, so $Y$ is a rational variety.
\end{proof}

Now let us consider a special case when $Z$ is coplanar. In this case, $S(\Ee)$ is a cubic surface with a unique singular point corresponding to the hyperplane containing $Z$, say $H$. Note that  $h^0(\Ee(1,1))=2$. Then there are 1-dimensional family of zero-dimensional subscheme $Z$ for which $\Ee$ fits into the sequence (\ref{seq1}). Such $Z$ should be contained in $C_H$. For each $Z$, we can consider $\PP_1$-family of hyperplanes that contain two points of $Z$ and this corresponds to a line contained in $S(\Ee)$. So we can find 6 lines contained in $S(\Ee)$ out of one such $Z$. As we vary $Z$ in the 1-dimensional family, we have infinitely many lines through $H$ contained in $S(\Ee)$. Thus we obtain the following statement:

\begin{proposition}
For the vector bundle $\Ee$ fitted into the sequence (\ref{seq1}) with coplanar $Z$, $S(\Ee)$ is a cone over a cubic curve in $\PP_2^*$ with the vertex point corresponding the hyperplane containing $Z$. 
\end{proposition}

\providecommand{\bysame}{\leavevmode\hbox to3em{\hrulefill}\thinspace}
\providecommand{\MR}{\relax\ifhmode\unskip\space\fi MR }
\providecommand{\MRhref}[2]{%
  \href{http://www.ams.org/mathscinet-getitem?mr=#1}{#2}
}
\providecommand{\href}[2]{#2}

\end{document}